\documentclass[10pt]{article}
\usepackage{amssymb}
\usepackage{latexsym}
\usepackage{graphicx}
\usepackage{subfigure}
\usepackage{epsfig}

\newtheorem{theorem}{Theorem}

\newtheorem{corollary}[theorem]{Corollary}

\newtheorem{example}{Example}

\newtheorem{lemma}[theorem]{Lemma}

\newenvironment{proof}[1][Proof]{\noindent\textbf{#1.} }{\ \rule{0.5em}{0.5em}}
\def\lab(#1)#2{\put(#1){\makebox(0,0)[c]{#2}}}

\begin{document}
\title{Relay Augmentation for Lifetime Extension of Wireless Sensor Networks}

\author{Marcus~Brazil, Charl~Ras, Doreen~Thomas}
\date{}
\maketitle

\renewcommand{\baselinestretch}{1.8}\normalsize

{\normalsize\begin{abstract}
We propose a novel relay augmentation strategy for extending the lifetime of a certain class of wireless sensor networks. In this class sensors are located at fixed and pre-determined positions and all communication takes place via multi-hop paths in a fixed routing tree rooted at the base station. It is assumed that no accumulation of data takes place along the communication paths and that there is no restriction on where additional relays may be located. Under these assumptions the optimal extension of network lifetime is modelled as the Euclidean $k$-bottleneck Steiner tree problem. Only two approximation algorithms for this NP-hard problem exist in the literature: a minimum spanning tree heuristic (MSTH) with performance ratio $2$, and a probabilistic $3$-regular hypergraph heuristic (3RHH) with performance ratio $\sqrt{3}+\epsilon$. We present a new iterative heuristic that incorporates MSTH and show via simulation that our algorithm performs better than MSTH in extending lifetime, and outperforms 3RHH in terms of efficiency.
\end{abstract}}

\section{Introduction}
Wireless sensor networks (WSNs) consist of small sensing devices that can be readily deployed in diverse environments to form distributed
wireless networks for collecting information in a robust and autonomous manner. Although early research was mainly motivated by potential
military uses, there are now many other important applications such as fault detection, environmental habitat monitoring, irrigation and terrain
monitoring (see, e.g., \cite{arampatzis}, \cite{mainwaring}). An example of the latter is the proposed use of WSNs to provide an early warning
system for bushfires in Australia. Deploying smart sensors in strategically selected areas can lead to early detection and an increased
likelihood of success in fire extinguishing efforts. Other applications include pollution control, climate control in large buildings, and
medical applications using implantable devices.

In this paper we assume that an initial WSN deployment phase has already taken place and that the locations of the sensors are known. In the relay augmentation phase a bounded number of relays are deterministically introduced to the sensing field, whereupon a fixed routing tree rooted at the base station and spanning all nodes is constructed. The power levels of the sensors and relays are then adjusted so that every node is able to transmit data to its parent in the routing tree. We assume that all sensors and relays transmit data at the same constant rate throughout the lifetime of the network.

In most WSNs the bulk of a sensor's energy consumption is attributable to communication. In turn, the energy consumed during data transmission is proportional to some power (usually between $2$ and $4$) of the communication distance. Therefore, under the above assumptions, the time till first node failure will be determined by the length of the longest edge in the routing tree, which we call the \textit{bottleneck edge}. There are many important WSN applications satisfying this property, including biomedical sensor networks \cite{cheng}, certain environmental monitoring networks, and natural disaster early detection systems. In many types of habitat monitoring applications, for example, the sensors are deployed at the locations of pre-determined data sources \cite{mainwaring}. Sparsity of the sensor locations may then necessitate the introduction of relays to obtain a connected network and to optimize the lifetime of the resulting sparse network \cite{xu}. Furthermore, many WSN applications either utilise simple aggregation functions based on maximum/minimum or average readings, or use advanced data fusion techniques to limit accumulation \cite{boulis},\cite{intan},\cite{krish}. Therefore our assumption of a uniform data transmission rate for all nodes is also realistic.

The problem of optimal relay augmentation for extending the lifetime (i.e., time till first node-death) of WSNs of the above type has previously been modelled as the \textit{$k$-bottleneck Steiner tree problem} ($k$-BST problem) or its dual, the \textit{minimum Steiner point tree problem} \cite{lloyd},\cite{brazil2},\cite{brazil3},\cite{bred}. The objective of the $k$-BST problem is, given a set of planar nodes (representing the sensors and called \textit{terminals}), to introduce at most $k$ new nodes (representing the relays and called \textit{Steiner points}) so that the resulting minimum spanning tree on the complete set of nodes minimises the length of the bottleneck edge. {As an example illustrating this problem consider Figure \ref{figBotDef}, where sensors are represented by black filled circles, relays by open circles, and bottleneck edges by grey lines. In Figure \ref{figBotDef}(a) no relays are added and the given sensors are connected by a \textit{minimum spanning tree} (MST). In Figure \ref{figBotDef}(b) two relays are included, which clearly gives a solution with shorter bottleneck edge than in Figure \ref{figBotDef}(a). In fact, the solution in Figure \ref{figBotDef}(b) is optimal for the given nodes when at most two Steiner points are allowed (note that other optimal solutions for the given terminals also exist).}

In this paper we introduce a new heuristic for this NP-hard problem, and provide simulation results that demonstrate a significant improvement in performance of our algorithm over the currently best performing algorithms. {In particular, we compare our algorithm to the \textit{minimum spanning tree heuristic} (MSTH) of ~Wang and Du \cite{wang}. Apart from making use of an exact algorithm for the $1$-BST problem developed by Bae et al.\cite{bae}, our algorithm also incorporates MSTH as a module.} In Section \ref{related} we look at related work. Section \ref{section1} formalises our network model and describes some properties of optimal solutions to the $k$-BST problem. In
Section \ref{section2} we provide necessary details of existing heuristic algorithms for solving the $k$-BST problem, as well as our new iterative $1$-BST heuristic algorithm, and provide an analysis, with examples, of this
algorithm in Section \ref{section3a}. Section \ref{section3} is devoted to an empirical demonstration of the performance of our algorithm.

\section{Related work}\label{related}
There are several constraints and network performance metrics - for instance coverage, power usage, lifetime and cost - that predetermine an
optimal WSN deployment strategy. The performance metric that we are interested in is lifetime since it is pertinent to such a large class
of WSN scenarios. We define the lifetime of a network as the time till first node failure, which is generally due to battery depletion.

Sensors are often overtaxed due to an uneven distribution of traffic flow. Many lifetime extending strategies in the literature therefore strive to dynamically adjust the routing topology in order to relieve the burden on
these sensors; see \cite{boua},\cite{kalp},\cite{zhang}. Unfortunately, these techniques are only germane if multiple available paths exist between nodes in the network. In WSNs with low connectivity or where routing protocols are limited by other application-specific or topological factors, deployment based strategies may be more valuable.

Often the most highly burdened sensors are within close proximity to the base station \cite{wang3}, hence the well-investigated objective of ``doughnut" or ``energy-hole" mitigation around the base station. Deploying additional relays close to the base station is one possible solution to the energy-hole problem, and this method has an analogue in density-varying random deployment strategies \cite{xin}. This approach, however, is only appropriate where there is a significant accumulation of data close to the base station. Given the assumption of a uniform data transmission rate for the networks dealt with in the current paper, these ``energy-hole algorithms" are not comparable to our algorithm.

Augmenting a network by deploying additional (non-sensing) relays is a powerful and relatively inexpensive method of optimising many topology
dependent WSN objectives. This concept is not new to the literature, and has been considered under various network models and objectives. Relay
augmentation strategies exist that supplement the primary objective of lifetime extension with objectives related to coverage \cite{du-2},
connectivity \cite{xu} and balanced traffic flow \cite{li},\cite{wang2}.

As mentioned above, here we model optimal lifetime extension as a $k$-bottleneck Steiner problem. This problem was introduced by Saraffzedeh and Wong \cite{sarrafzadeh}, and various authors have considered it in the context of facility location science \cite{drezner},\cite{love}. The problem is NP-hard; in fact, unless P=NP, no polynomial-time
algorithm exists for the problem in the Euclidean plane with a \textit{performance ratio} less than $\sqrt{2}$, {where this ratio is defined as the theoretically largest possible value attained when dividing the length of the bottleneck edge produced by a given polynomial-time algorithm by the length of the bottleneck edge in an optimal solution}. ~Wang and Du, the authors who
demonstrated this complexity bound, present in \cite{wang} the first deterministic approximation algorithm for the $k$-bottleneck problem in
both the rectilinear and Euclidean planes. They also prove that the performance ratio of their algorithm, the afore-mentioned minimum spanning tree
heuristic, is bounded above by $2$ for each of these metrics. Du et al. \cite{du1} describe a probabilistic $3$-regular hypergraph heuristic ($3$RHH) with performance ratio $\sqrt{3}+\epsilon$ which runs in $\frac{1}{\epsilon}\times\mathrm{poly}(n,k)$ time, where $n$ is the number of given terminals. Recently Bae et al. \cite{bae1},\cite{bae} developed exact algorithms with exponential complexity.

\section{Sensor Network Model and Properties}\label{section1}
The energy consumption for each node in a WSN can be modelled, as in \cite{wang2}, by $$\mathbb{E}(r)=r^\alpha + c$$ where $r$ is packet transmission distance, $c$ is a constant
and is small relative to $r^\alpha$, and $\alpha\in [2,4]$. The power consumption attributable to packet receipt is assumed to be negligible. Let $\ell_{\max}(T)$ denote the length of the longest edge in a tree $T$. For any set $X$ of $n$ embedded terminals, let $\mathcal{C}(X,k)$ be the
set of all trees interconnecting $X$ and at most $k$ other points in the plane.

\noindent \textbf{Definition.} The \emph{$k$-bottleneck Steiner tree problem}, or $k$-BST problem, has the following form:
\begin{quote}\begin{description}
    \item[\textbf{Given}] A set $X$ of $n$ points (terminals) in $\mathbb{R}^2$, and a positive integer $k$.
    \item[\textbf{Find}] A set $S$ (the Steiner points) of at most $k$ points in $\mathbb{R}^2$, and a spanning tree $T$ on $X\cup S$ such that
    $\ell_{\max}(T) \leq \ell_{\max}(T^{\,\prime})$ for any $T^{\,\prime}\in \mathcal{C}(X,k)$.
\end{description} \end{quote}

We refer to $T$ as a \textit{minimum $k$-bottleneck Steiner tree ($k$-MBST)}. In general, $k$-MBSTs are not easy to construct. {The following results, culminating in Corollary \ref{deg5MST}, allow us to restrict the search-space somewhat.} The next lemma from \cite{brazil} is an extension of the Swapping Algorithm found in \cite{lee}. It is a consequence of the matroid properties of minimum spanning trees.

\begin{lemma}\label{MSTswap}Let $T$ be a minimum spanning tree on the terminal set $X$, and let $T^{\,\prime}$ be a spanning tree for $X$.
We can transform $T^{\,\prime}$ to $T$ by a series of edge swaps, where each swap involves replacing an edge $e_i \in E(T^{\,\prime})$ by $e_j
\in E(T)$ such that $\vert e_i\vert \geq \vert e_j\vert$, and at each stage the graph is a spanning tree.
\end{lemma}

\begin{corollary}\label{MST-coroll}If $T^{\,\prime}$ is a $k$-bottleneck Steiner minimum tree on $X$ with Steiner points $S$ then every minimum
spanning tree on $X\cup S$ is a $k$-bottleneck Steiner minimum tree on $X$.
\end{corollary}
\begin{proof}
Let $T$ be a minimum spanning tree on $X\cup S$. By Lemma~\ref{MSTswap}, we can transform $T^{\,\prime}$ to $T$ by a series of edge swaps, each
of which replaces an edge with another of the same length or shorter.
\end{proof}

{It is well-known in the literature that for any set of planar nodes an Euclidean MST exists with highest degree (i.e., the number of edges meeting at a node) at most $5$. We therefore have the following result.}

\begin{corollary}\label{deg5MST}For any set of terminals there exists a $k$-MBST that is an MST on its complete set of nodes and such that the maximum node degree is at most $5$.
\end{corollary}

{As recognized in \cite{bae}, the \textit{smallest spanning circle} concept is fundamental to the $k$-BST problem. A smallest spanning circle for a given set of nodes is a circle of smallest radius that includes all nodes on or within its boundary. The centre of a smallest spanning circle minimises the largest distance to a node of the given set. Therefore we may assume that every Steiner point in a $k$-MBST is located at the centre of the smallest spanning disk of its neighbours (for if not then we could relocate the Steiner point to this location without increasing the length of the bottleneck); see Figure \ref{figchebyCent}.}

The above properties now provide us with a simple (albeit time expensive) exact solution to the $1$-BST problem: we select the
cheapest tree after iterating through all subsets of at most five neighbours for the Steiner point and, for each iteration, locating the Steiner point at the centre of the smallest spanning disk on the subset and then constructing an MST on all nodes. The total complexity of this algorithm is therefore $\mathcal{O}(n^6\log n)$. Fortunately the $1$-BST can be solved in $O(n\log n)$ time by using the algorithm of Bae et al. \cite{bae}.

\section{Algorithms}\label{section2}
Before providing the details of our main algorithm we briefly describe the algorithms MSTH and $3$RHH. These algorithms will be compared to ours via simulation in Section \ref{section3}.

\subsection{The Minimum Spanning Tree Heuristic}
The only deterministic algorithm in the literature for approximating a $k$-MBST is the minimum spanning tree heuristic (MSTH) of Wang and Du
\cite{wang}. The authors prove that the performance ratio of MSTH lies inclusively between $\sqrt{3}$ and $2$, and they conjecture that, if
$P\neq NP$, $\sqrt{3}$ is the best possible performance ratio of any polynomial-time algorithm for the Euclidean $k$-BST problem. The time-complexity of MSTH is $O(n\log n)$. In MSTH all Steiner points are of degree two and are referred to as \textit{beads}. Let $T$ be an MST on any $q$ nodes. We define \textit{beading} an
edge $e_i$ of $T$ with $n_i$ beads as placing $n_i$ equally spaced beads along $e_i$. Our main
algorithm makes extensive use of MSTH as a ``look-ahead" component, which we denote by BEAD. Subroutine BEAD is virtually identical to MSTH,
except that in MSTH an MST is not part of the input (only the locations of the terminals are). For
each $e_i$ in $T$ let $l(e_i)=\displaystyle\frac{\vert e_i\vert}{n_i+1}$.\\

\noindent\textbf{Algorithm 1} Subroutine BEAD (MSTH)\\
\noindent\textbf{Input}: An MST $T$ on $q$ nodes and a positive integer $j$\\
\noindent\textbf{Output}: A beaded MST $T^{\,\prime}$\\
\textbf{1}: Clear all beads of $T$, i.e., set $n_i=0$ for any $i$\\
\textbf{2}: $e_1,...,e_{q-1}$ be the edges of $T$\\
\textbf{3}: Compute $l(e_i)$ for each $e_i$\\
\textbf{4}: Sort the edges in non-decreasing order of $l(\cdot)$\\
\textbf{5}: Add a bead to $e_i$ of largest $l(\cdot)$ value\\
\textbf{6}: Update $l(e_i)$\\
\textbf{7}: Relocate the beads on $e_i$ so that they are equally spaced\\
\textbf{8}: Reset $e_i$'s position in the ordering\\
\textbf{9}: Repeat Steps 5-8 until $j$ beads have been added\\

It is demonstrated in \cite{wang} that Algorithm 1 is the best possible solution when all Steiner points are of degree two.

\subsection{The $3$-Regular Hypergraph Heuristic}
In this section we discuss the only other heuristic for the $k$-BST problem, namely the randomized algorithm of Du et al. \cite{du1}. The core of their algorithm is based on a randomized approximation scheme for finding minimum spanning trees in weighted $3$-uniform hypergraphs discovered by Promel and Steger \cite{prom}. This MST component of $3$RHH, say MST$3$H, is the least efficient part of Du et al.'s algorithm, and we will need to discuss certain technical aspects of MST$3$H in order to demonstrate the superiority of our algorithm over $3$RHH.

As input to MST$3$H, a weighted $3$-regular hypergraph containing $O(n^3)$ edges and with edge weights at most $k$ is constructed by $3$RHH. The edge-weights are then randomized by MST$3$H to a size of $O(mnk)$, where $m$ is the number of edges of the hypergraph. A matrix $A$ is then constructed containing entries of order $O(2^w)$, where $w$ is some edge-weight. The determinant of $A$ is calculated and is of size $O((2^w)^n)$. The determinant is therefore of order $2^{O(kn^5)}$. To find an MST in the given hypergraph the largest $w$ such that $2^{2w}$ divides the determinant must be calculated. Due to the extremely large size of the determinant this takes very long in practice.

As an example we implemented an instance with $n=10$ and $k=3$. Then $3$-RHH produced a hypergraph with a determinant containing $496,631$ digits, which made it practically intractable to find the relevant divisor as stipulated by the algorithm. The size of the determinant in this example is typical for instances with $n=10$.

\subsection{A Naive Iterative 1-Bottleneck Steiner Tree Heuristic}
A naive version of the iterative 1-BST heuristic (or I1-BSTH) on a set $X$ of terminals might proceed as follows. An optimal $1$-BST on $X$ is constructed by Bae et al.'s algorithm. This sequence of routines is repeated $k-1$ times, i.e.,
until $k$ new points have been added.

In order to compare the performance of this heuristic against MSTH we implemented both in C (with minor modifications as described below in
Subroutine CONVERT), and tested them on uniformly distributed terminal sets. Roughly speaking, simulations show a maximum average improvement of
approximately $4\%$ in maximum edge length for naive I1-BSTH over MSTH. This average can be increased to about $5\%$ by combining naive I1-BSTH and MSTH into a
meta-heuristic, which shows that MSTH can outperform naive I1-BSTH on some terminal sets. A simple example where this occurs is described next,
and is illustrated in Figures \ref{fig4} and \ref{fig2}; in these figures (and throughout) the solid circles are terminals and the open circles
are Steiner points.

\begin{example}\label{eg1}Let $k=2$ and consider four terminals with coordinates $\mathbf{t_1}=(2.0, 9.1),\ \mathbf{t_2}=(3.0, 8.6),\
\mathbf{t_3}=(4.6, 3.1),\ \mathbf{t_4}=(8.6, 9.2)$. Then $p_1=2.85$ and $p_2=3.63$ represent the longest edges in the solutions of
MSTH and naive I1-BSTH respectively.
\end{example}

Fig. \ref{fig2} also illustrates one of the modifications made to naive I1-BSTH, namely that of correcting any bent edges that may appear in the
final solution; a bent edge actually consists of two edges incident, at an angle of less than $180^\circ$, to a Steiner point of degree two.

\subsection{The Pre-beaded Iterative $1$-Bottleneck Steiner Tree Heuristic}
Fig. \ref{fig2} highlights a major draw-back of iterative greedy heuristics: an optimal solution at some iteration may undercut the gains of
subsequent iterations. We notice that, bent-edge correction aside, the tree from Fig. \ref{fig2c} provides no decrease in maximum edge-length
over Fig. \ref{fig2a}. This shortcoming permeates the class of solutions constructed by naive I1-BSTH. A look-ahead routine is therefore
warranted as a modification to 1I-BSTH, but only if gains in performance are made by such a routine over the afore-mentioned meta-heuristic. The
pre-beaded iterative $1$-BST heuristic (pre-beaded I1-BSTH), described in Algorithm 2, utilises subroutine BEAD for this
purpose.\\

\noindent\textbf{Algorithm 2} Pre-beaded I1-BSTH\\
\noindent\textbf{Input}: A set $X$ of $n$ points in the plane, and a positive integer $k$\\
\noindent\textbf{Output}: An approximate $k$-MBST $T$ interconnecting $X$ and $k$ new points\\
\textbf{1}: Construct a minimum spanning tree $T_0$ on $X$\\
\textbf{2}: Let $p=c=0$\\
\textbf{3}: \textbf{while} $c < k$ \textbf{and} $p < k$\\
\textbf{4}:$\hspace{0.25in}$ Run BEAD with input $T_c$ and $k-1-p$, and output $T_c^{\,\prime}$\\
\textbf{5}:$\hspace{0.425in}$ Run Bae's algorithm for $1$-BST on $T_c^{\,\prime}$ with output $T^{\,\prime}$ and $\mathbf{s}$\\
\textbf{6}:$\hspace{0.425in}$ Run CONVERT with input $T^{\,\prime}$ and $\mathbf{s}$, and output $T^{\,\prime\prime}$\\
\textbf{7}:$\hspace{0.425in}$ Let $p^{\,\prime}$ be the number of non-bead Steiner points of $T^{\,\prime\prime}$\\
\textbf{8}:$\hspace{0.425in}$ Run BEAD with input $T^{\,\prime\prime}$ and $k-p^{\,\prime}$, and output $T^*$\\
\textbf{9}:$\hspace{0.175in}$ \textbf{end for}\\
\textbf{10}:$\hspace{0.175in}$ Increment $c$ by $1$, let $T_c$ be the cheapest $T^*$ produced\\
\textbf{11}:$\hspace{0.175in}$ Let $p$ be the number of non-bead Steiner points of $T_c$\\
\textbf{12}: \textbf{end while}\\

Pre-beaded I1-BSTH also contains the procedure CONVERT, which consists exactly of the modifications that were made to naive I1-BSTH for its
implementation in C.\\

\noindent\textbf{Algorithm 3} Subroutine CONVERT\\
\noindent\textbf{Input}: A tree $T$ interconnecting a set of terminals $X$ and a set of Steiner points $S$, and an element $\mathbf{s}\in S$\\
\noindent\textbf{Output}: The modified tree $T^{\,\prime}$\\
\textbf{1}: Delete all degree-one Steiner points of $T$\\
\textbf{2}: Correct all bent edges and relocate beads to be equally spaced along edges\\
\textbf{3}: Relocate $\mathbf{s}$ at the Chebyshev point of its neighbours, and add the edges incident to $\mathbf{s}$\\

\section{Analysis of Pre-beaded I1-BSTH}\label{section3a}
There are three aspects of pre-beaded I1-BSTH that we examine more closely in this section. The first is the look-ahead component
actualised by the call to BEAD in Line 12 of Algorithm 2. Calculating an optimal sequence of Steiner point additions can generally not be done
greedily, and calculating every potential sequence would take us into exponential complexity. However, we can remain within polynomial
complexity if we assume, at any given iteration of the algorithm, that the remaining sequence of additions will consist of beads only. This
assumption can be viewed as providing a lower bound for the overall (including past and future iterations) performance of an optimal Steiner
point addition sequence. This lower bound is clearly equal to the performance of MSTH, and hence we are guaranteed that pre-beaded I1-BSTH never
performs worse than MSTH. This result on its own is not satisfactory, since the meta heuristic combining the naive version of I1-BSTH and MSTH
has the same property. But as will be demonstrated in the next section, algorithm pre-beaded I1-BSTH performs significantly better than the meta
heuristic.

We use Fig. \ref{figEg1} to illustrate the benefit of the look-ahead component: we compare the solutions constructed by naive I1-BSTH and
pre-beaded I1-BSTH on a set of six terminals $\mathbf{t_1}=(968.4, 506.4)$, $\mathbf{t_2}=(3.9, 86.8)$, $\mathbf{t_3}=(188.8, 7.5)$,
$\mathbf{t_4}=(779.2, 675.9)$, $\mathbf{t_5}=(238.1, 644.4)$, $\mathbf{t_6}=(620.6, 2.4)$, and where $k=2$. The three longest edges of the MST on these
terminals are (in descending order of length) $e_1,e_2,e_3$; see Fig. \ref{figEga}.

The naive algorithm, Fig. \ref{figEgb}, selects $\{\mathbf{t_1},\mathbf{t_4},\mathbf{t_5},\mathbf{t_6}\}$ as the optimal neighbour set for the
first Steiner point, $\mathbf{s}_1$. The Chebyshev centre of this set is defined by the subset $\{\mathbf{t_1},\mathbf{t_5},\mathbf{t_6}\}$, and
the subsequent tree, after updating the MST to include $\mathbf{s_1}$, has $e_3$ as its longest edge. The second iteration's optimal point
happens to be a bead subdividing $e_3$, and the longest edge of the final tree is therefore $e_4=\mathbf{s_1t_6}$ (equivalently
$\mathbf{s_1t_5}$ or $\mathbf{s_1t_1}$) and has a length of 389.87 units. Fig. \ref{figEgc} shows the final tree produced by pre-beaded
I1-BSTH. Here one may view the algorithm as initially adding a single bead to each of edges $e_1$ and $e_2$; it then optimally repositions the
bead on $e_1$ to create the Steiner point $\mathbf{s_1^\prime}$ with neighbour-set $\{\mathbf{t_3},\mathbf{t_5},\mathbf{t_6}\}$. Finally it
looks for an improved position for the bead on $e_2$, but determines that that bead is already an optimal Steiner point with respect to the
other nodes. The resulting longest edge is $e_5$, with $\vert e_5\vert =374.46$. Note also that naive I1-BSTH greedily deletes the two longest
edges, $e_1,e_2$, in its first iteration, resulting in a final tree with relatively long edges incident to the first Steiner point. Pre-beaded
I1-BSTH, however, deletes edges $e_1$ and $e_3$ in the first iteration because it assumes that $e_2$ will at least be beaded in the second
iteration. The result is a Steiner point $\mathbf{s_1^\prime}$ with relatively shorter incident edges. This example is simple, but typical of
many cases where the pre-beaded algorithm outperforms the naive version.

The second aspect we would like to examine concerns the call to BEAD in Line 4. This call ensures that there are $k-1$ Steiner points in the
input tree before any main iteration. Consequently our algorithm may be viewed as beading the entire initial MST, and then one-by-one removing
beads and adding them back optimally. It is possible to construct a version of I1-BSTH that does not utilise this BEAD call but still contains
the look-ahead routine. This so called \textit{Post-beaded} I1-BSTH  performs very similarly to our pre-beaded algorithm, but falls short in certain
instances, of which Fig. \ref{figEg2} is a case in point.

Before the first iteration of pre-beaded I1-BSTH the MST in Fig. \ref{figEg2a} already
contains a bead on the longest edge $\mathbf{t_1t_4}$. As opposed to the post-beaded version, our algorithm is able to utilise this bead as a
potential neighbour for the Steiner point $\mathbf{s_1}$, and the result is the simultaneous introduction of two Steiner points during the first
iteration.

\section{Empirical Results}\label{section3}
Multiple $k$-BST problem instances, \textit{uniformly distributed} (i.e., randomly with a uniform distribution) in a 10,000 by 10,000 point region, were generated and solved in MATLAB for MSTH, post-beaded I1-BSTH, and
pre-beaded I1-BSTH. We focussed on terminal sets of size $n=600$ and $k=20,40,...,200$ and performed $300$ simulations for each case. Tables \ref{newtable1}--\ref{newtable3}, and Figures \ref{figNew1} and \ref{figNew3} present the simulation results. We also performed $300$ simulations for instances with terminals distributed non-uniformly in a 10,000 by 10,000 point region, with results depicted in Figure \ref{figNewD} (once again we focussed on $n=600$ and $k=20,40,...,200$). In particular, for this last group of simulations we selected a random point to represent a base station and then distributed the points with an increasing density towards the base station. The run times for our algorithm are not as important as algorithms which are dynamic in nature (such as for instance routing protocols). However, it is not unreasonable to envision sensor network applications arising in the future that have many more terminals than just a few thousand; in other words we must still show that our algorithms scale well. We have therefore also included run time results in our simulations. In Table \ref{newtable4} we present average run time results for $3$RHH on instances of up to $6$ terminals and $3$ Steiner points.

\begin{table}[h!t!]
\caption{MSTH\label{newtable1}}
\begin{center}
\begin{tabular}{l|l|l|l}
\hline \hline
\textbf{$\hspace{0.4cm}$ $k$} & $\hspace{0.3cm}$\textbf{Avg. lifetime (s) $\alpha=2$} & $\hspace{0.25cm}$\textbf{Avg. lifetime (s) $\alpha=4$} & $\hspace{0.2cm}$\textbf{Avg. run time (s)} \\
\hline
$20$ & 19.25 & 0.037 & 1.05\\
$40$ & 20.12 & 0.040 & 1.12\\
$60$ & 21.37 & 0.046 & 1.28\\
$80$ & 23.74 & 0.056 & 1.29\\
$100$ & 24.65 & 0.061 & 1.33\\
$120$ & 25.68 & 0.066 & 1.44\\
$140$ & 26.64 & 0.071 & 1.47\\
$160$ & 28.50 & 0.081 & 1.51\\
$180$ & 30.95 & 0.096 & 1.58\\
$200$ & 31.89 & 0.102 & 1.68\\
\hline \hline
\end{tabular}
\end{center}
\end{table}

\begin{table}[h!t!]
\caption{Post-beaded I1-BSTH\label{newtable2}}
\begin{center}
\begin{tabular}{l|l|l|l}
\hline \hline
\textbf{$\hspace{0.4cm}$ $k$} & $\hspace{0.3cm}$\textbf{Avg. lifetime (s) $\alpha=2$} & $\hspace{0.25cm}$\textbf{Avg. lifetime (s) $\alpha=4$} & $\hspace{0.2cm}$\textbf{Avg. run time (s)} \\
\hline
$20$ & 20.53 & 0.042 & 1.88\\
$40$ & 22.19 & 0.049 & 2.31\\
$60$ & 24.31 & 0.059 & 2.97\\
$80$ & 27.11 & 0.073 & 3.25\\
$100$ & 28.58 & 0.082 & 3.88\\
$120$ & 30.44 & 0.092 & 4.41\\
$140$ & 31.06 & 0.096 & 4.89\\
$160$ & 35.66 & 0.127 & 5.39\\
$180$ & 38.63 & 0.149 & 5.81\\
$200$ & 40.55 & 0.164 & 6.33\\
\hline \hline
\end{tabular}
\end{center}
\end{table}

\begin{table}[h!t!]
\caption{Pre-beaded I1-BSTH\label{newtable3}}
\begin{center}
\begin{tabular}{l|l|l|l}
\hline \hline
\textbf{$\hspace{0.4cm}$ $k$} & $\hspace{0.3cm}$\textbf{Avg. lifetime (s) $\alpha=2$} & $\hspace{0.25cm}$\textbf{Avg. lifetime (s) $\alpha=4$} & $\hspace{0.2cm}$\textbf{Avg. run time (s)} \\
\hline
$20$ & 21.81 & 0.048 & 1.98\\
$40$ & 21.56 & 0.047 & 2.22\\
$60$ & 26.34 & 0.069 & 3.02\\
$80$ & 28.20 & 0.080 & 3.20\\
$100$ & 29.08 & 0.085 & 3.89\\
$120$ & 30.88 & 0.095 & 4.49\\
$140$ & 32.77 & 0.107 & 4.77\\
$160$ & 35.37 & 0.125 & 5.29\\
$180$ & 40.24 & 0.162 & 6.12\\
$200$ & 44.35 & 0.197 & 6.56\\
\hline \hline
\end{tabular}
\end{center}
\end{table}

\begin{table}[h!t!]
\caption{Run times for $3$RHH\label{newtable4}}
\begin{center}
\begin{tabular}{l|l|l}
\hline \hline
\textbf{$\hspace{0.4cm}$ $n$} & $\hspace{0.3cm}$ $k$ & $\hspace{0.2cm}$\textbf{Avg. run time (s)} \\
\hline
$3$ & 2 & 23.52\\
$4$ & 2 & 180.41\\
$5$ & 3 & 7653.23\\
$6$ & 3 & 28845.49\\

\hline \hline
\end{tabular}
\end{center}
\end{table}

In our simulations we observed that the performance of pre-beaded I1-BSTH relative to MSTH roughly increases from $k=0$, reaching a peak in the approximate range $0.2 \leq \frac{k}{n}\leq 0.34$, whereafter performance decreases at a decreasing rate. We also observed near-asymptotic behaviour as $k$ increases beyond $k=n$ (see Figure \ref{figNew4}, where $\ell^k_{\mathrm{BSTH}}$ and $\ell^k_{\mathrm{MSTH}}$ are defined below). The limiting behaviour might not continue indefinitely, but we would like to state the following lemma and corollary.

For any set $X$ of $n$ terminals let $\ell_{\mathrm{OPT}}^k(X)$, $\ell_{\mathrm{MSTH}}^k(X)$, and $\ell_{\mathrm{BSTH}}^k(X)$ represent the
length of the longest edge in three potential solutions to the $k$-BST problem on $X$; namely the optimal solution, the solution generated by
MSTH, and the solution generated by pre-beaded I1-BSTH respectively. Let $L_{\mathrm{SMT}}(X)$ be the total length of the classical \textit{Steiner minimal tree}
(SMT) on $X$ (i.e., the shortest total length tree spanning $X$ where any number of Steiner points may be introduced), and $L_{\mathrm{MST}}(X)$ be the total length of the MST on $X$. We let
$\ell(X,\mathrm{type1},\mathrm{type2})=\displaystyle\lim_{k\rightarrow\infty}\left(\ell_{\mathrm{type1}}^k(X)/\ell_{\mathrm{type2}}^k(X)\right)$,
where $\mathrm{type1},\mathrm{type2}\in\{\mathrm{OPT},\mathrm{MSTH},\mathrm{BSTH}\}$.

\begin{lemma}$\ell(X,\mathrm{MSTH},\mathrm{OPT})=L_{\mathrm{MST}}(X)/L_{\mathrm{SMT}}(X)$.
\end{lemma}
\begin{proof}As $k\rightarrow \infty$, the length of every edge in the optimal beading of the MST on $X$ approaches $\ell_{\mathrm{MSTH}}^k(X)$,
and hence $k.\ell_{\mathrm{MSTH}}^k(X)\rightarrow L_{\mathrm{MST}}(X)$. Similarly, as $k\rightarrow \infty$ in the optimal bottleneck tree, at
most $n-2$ Steiner points have degree greater than $2$, and hence almost all Steiner points are beads. So again $k.\ell_{\mathrm{OPT}}^k(X)$
approaches the total length of the minimum length tree interconnecting $X$, namely the SMT on $X$.
\end{proof}

Let $\mathrm{E}(\cdot)$ be the expected value of some parameter determined by $X$, where $X$ is uniformly distributed, and let $\mathrm{E}_0
=\mathrm{E}\left(L_{\mathrm{MST}}(X)/L_{\mathrm{SMT}}(X)\right)$. Gilbert and Pollack \cite{gilbert} provide an upper-bound of $1.064$ for
$\mathrm{E}_0$.

\begin{corollary}If $\mathrm{E}\left(\ell(X,\mathrm{BSTH},\mathrm{MSTH})\right)=p$ then\\
$\mathrm{E}\left(\ell(X,\mathrm{BSTH},\mathrm{OPT})\right)=p.\mathrm{E}_0$
\end{corollary}
\begin{proof}\\
$\mathrm{E}\left(\ell(X,\mathrm{BSTH},\mathrm{OPT})\right)$\\
$=\mathrm{E}\left(\ell(X,\mathrm{MSTH},\mathrm{OPT})\cdot\ell(X, \mathrm{BSTH},\mathrm{MSTH})\right)$\\
$=\mathrm{E}\left(\ell(X,\mathrm{MSTH},\mathrm{OPT})\right)\cdot\mathrm{E}\left(\ell(X, \mathrm{BSTH},\mathrm{MSTH})\right)$\\
$=\mathrm{E}\left(L_{\mathrm{MST}}(X)/L_{\mathrm{SMT}}(X)\right)\cdot p$.
\end{proof}
$\ $\\

{Figure \ref{figNew4} suggests that $\mathrm{E}\left(\ell(X, \mathrm{BSTH},\mathrm{MSTH})\right)$ is approximately $0.985$. Based on this value, and by using the above corollary, we can estimate the performance ratio of I1-BSTH on uniformly distributed sets as $k$ tends to infinity: $\mathrm{E}\left(\ell(X,\mathrm{BSTH},\mathrm{OPT})\right)=p.\mathrm{E}_0\approx 0.985\times \mathrm{E}_0\leq 1.04804$, where the final inequality comes from the above mentioned upper bound on $\mathrm{E}_0$ given in \cite{gilbert}.}

\section{Conclusion}\label{section4}
One can model the problem of the optimal lifetime extension of homogeneous WSNs as the Euclidean $k$-bottleneck Steiner tree problem. In this
paper we propose a new approximation algorithm for the $k$-BST problem, namely the pre-beaded iterative $1$-BST heuristic, and test its
performance relative to the minimum spanning tree heuristic (MSTH) of Wang and Du, which is currently the best performing heuristic for
the problem. By incorporating MSTH as a ``look-ahead" component and utilising Bae et al.'s algorithm for solving the $1$-bottleneck Steiner tree problem, we develop an algorithm that significantly outperforms MSTH for uniformly distributed terminal sets. We also demonstrate that the probabilistic $3$-regular hypergraph heuristic ($3$RHH) of Du et al. is impractical even for small $n$. To our knowledge this is the
first well-performing tractable approach to solving the $k$-BST problem.

\section{Acknowledgement}
The authors would like to acknowledge the Australian Research Council for financially supporting this research through a Discovery grant.

\clearpage
\begin{figure}[htb]
    \begin{center}
        \subfigure[A minimum spanning tree on six given nodes (sensors)]{\includegraphics[scale=0.3]{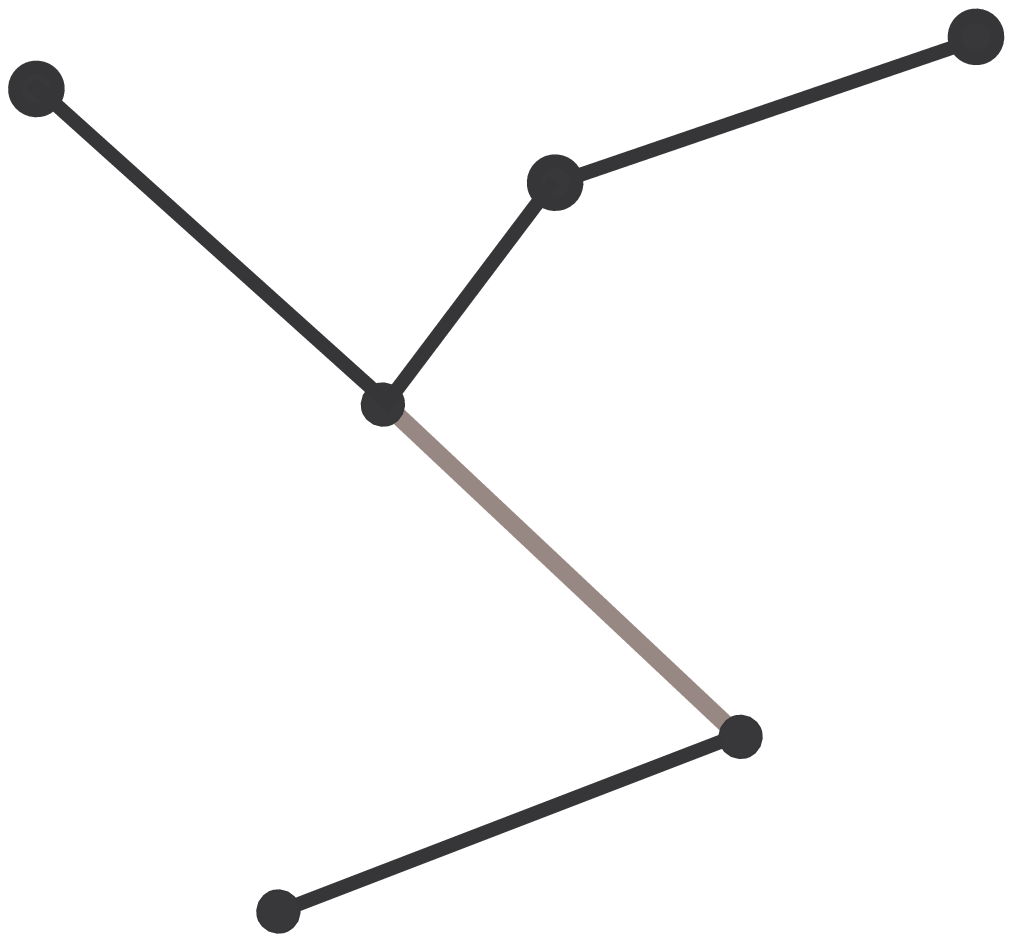}}
        \subfigure[A spanning tree on the given nodes which includes two Steiner points (relays)]{\includegraphics[scale=0.3]{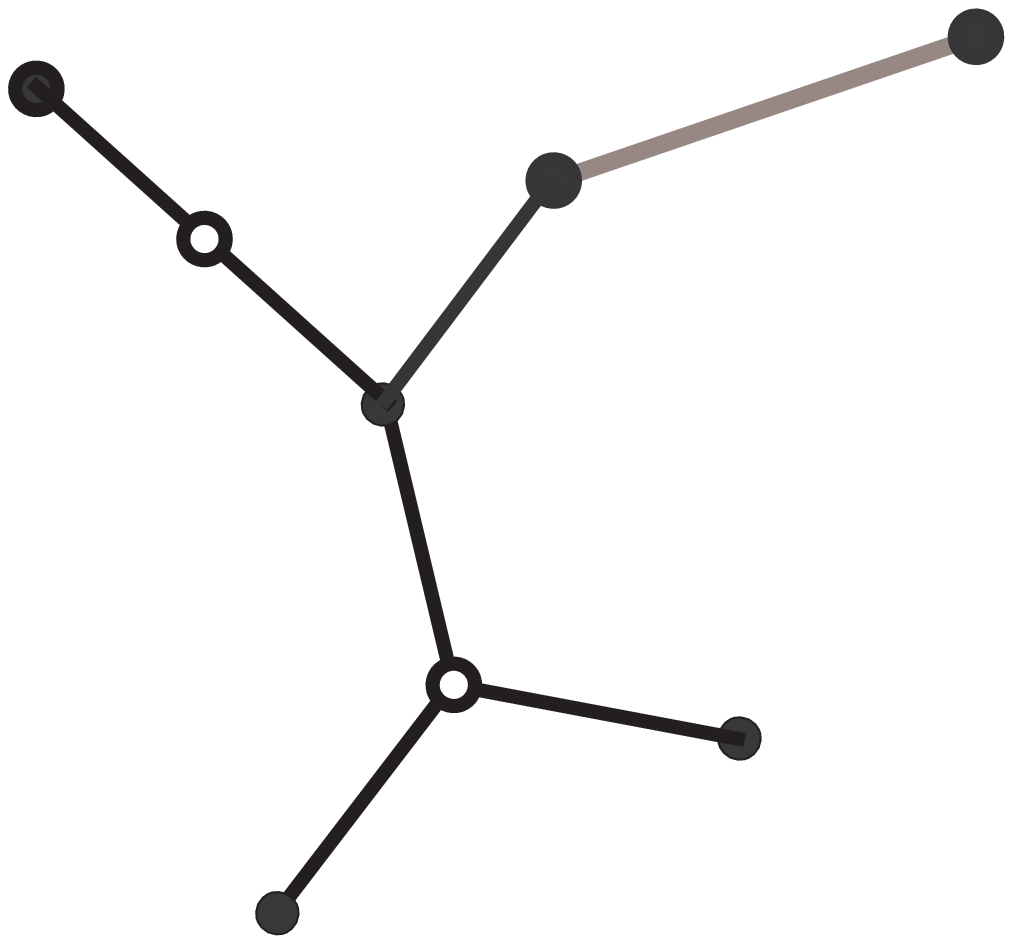}}
    \end{center}
    \caption{Two potential solutions to an instance of the $2$-bottleneck Steiner tree problem \label{figBotDef}}
\end{figure}

\clearpage
\begin{figure}[htb]
    \begin{center}
        \subfigure[]{\includegraphics[scale=0.3]{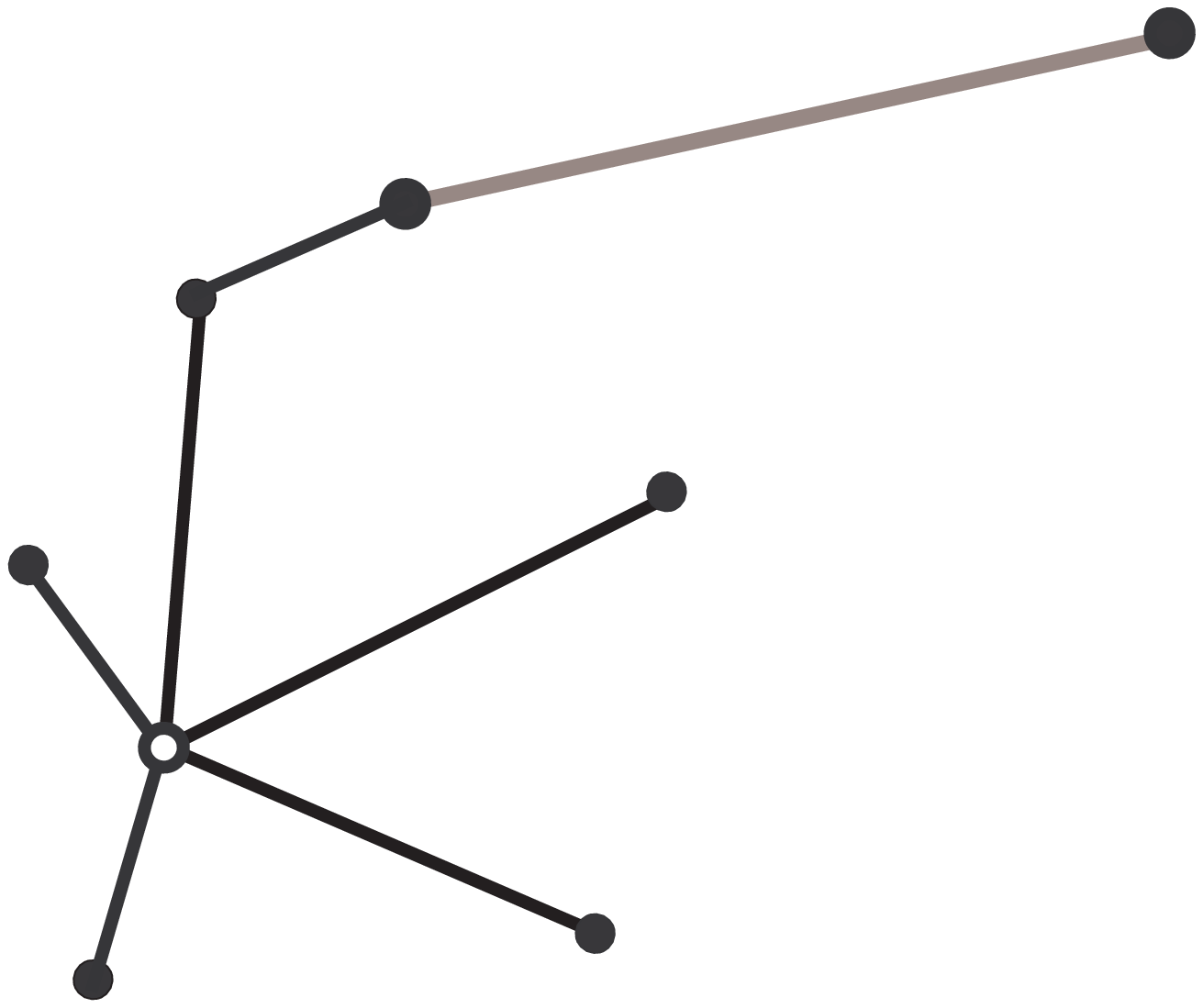}}
        \subfigure[]{\includegraphics[scale=0.3]{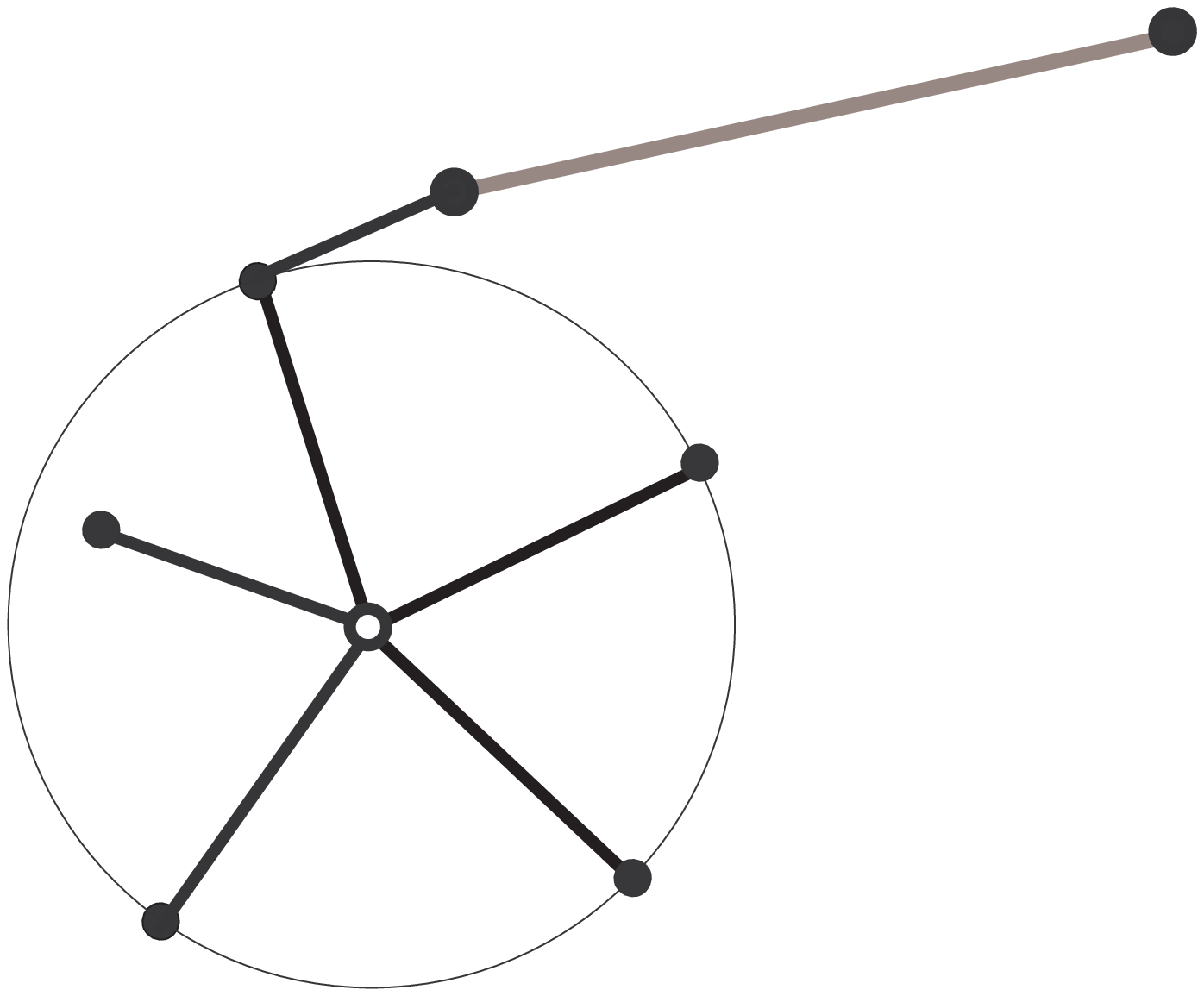}}
    \end{center}
    \caption{Steiner points may be relocated, without increasing the bottleneck, so that they lie at the centre of the smallest spanning circle of their neighbours \label{figchebyCent}}
\end{figure}

\clearpage
\begin{figure}[htb]
    \begin{center}
        \subfigure[Four given terminals plus their MST]{\includegraphics[scale=0.3]{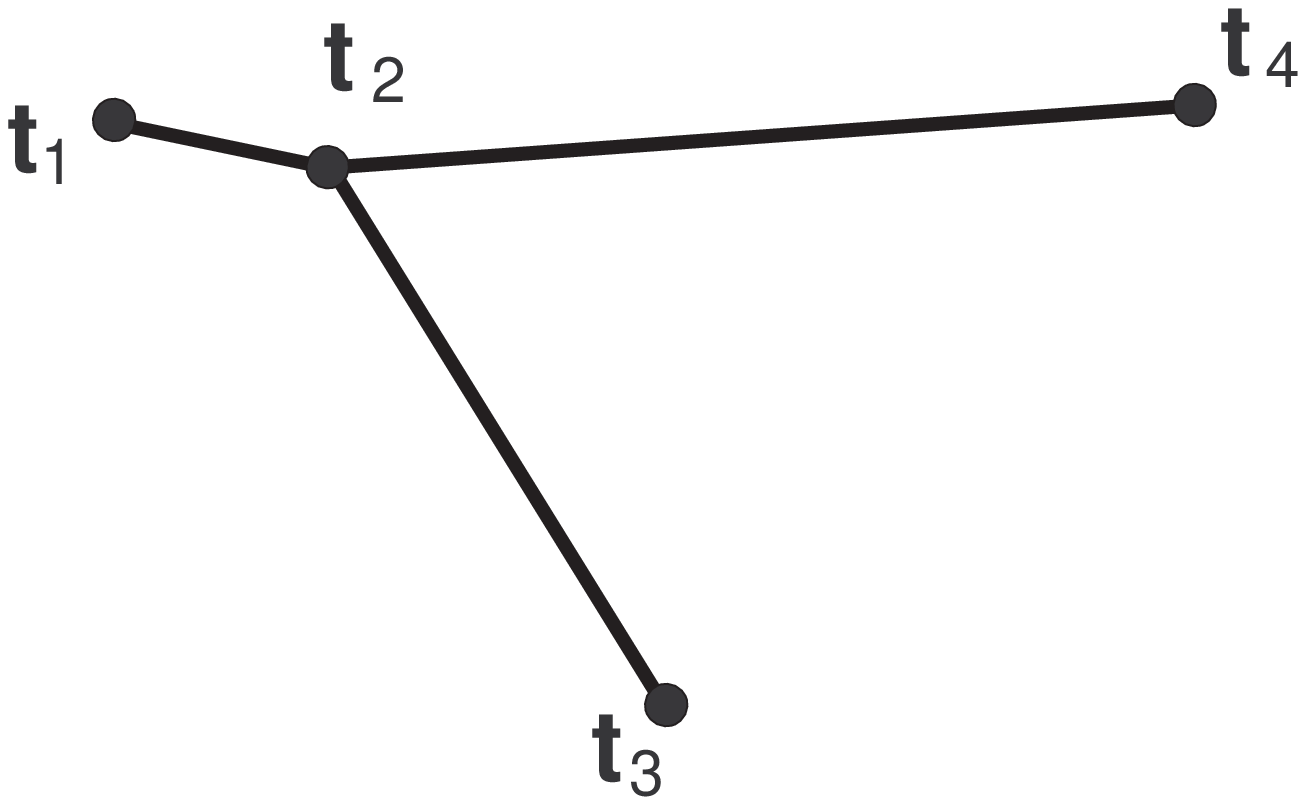}}
        \subfigure[Longest edges are beaded]{\includegraphics[scale=0.3]{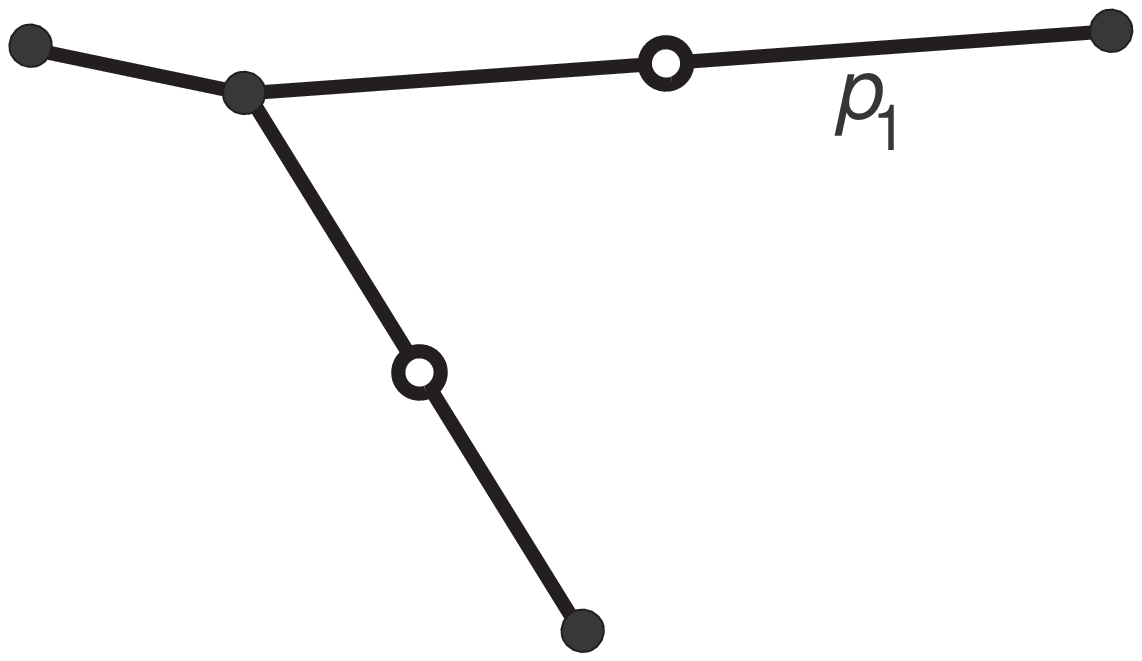}}
    \end{center}
    \caption{The solution constructed by MSTH as described in Example \ref{eg1}
    \label{fig4}}
\end{figure}

\clearpage
\begin{figure}[htb]
  \begin{center}
    \subfigure[First Steiner point is added]{\label{fig2a}\includegraphics[scale=0.3]{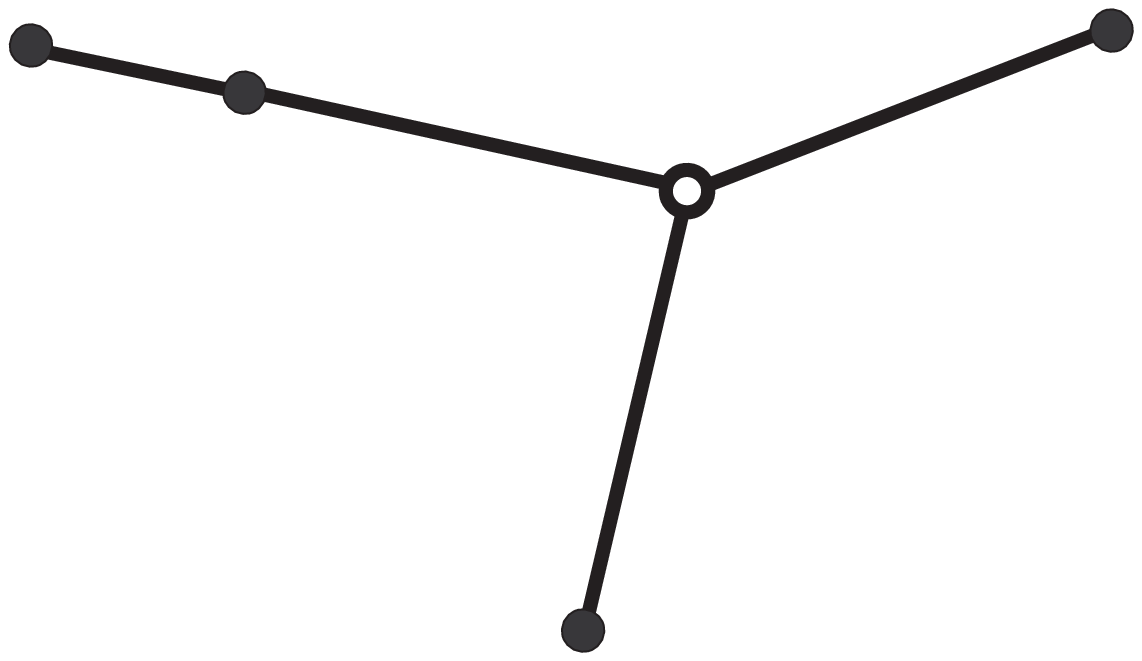}}
    \subfigure[Second Steiner point is added]{\includegraphics[scale=0.3]{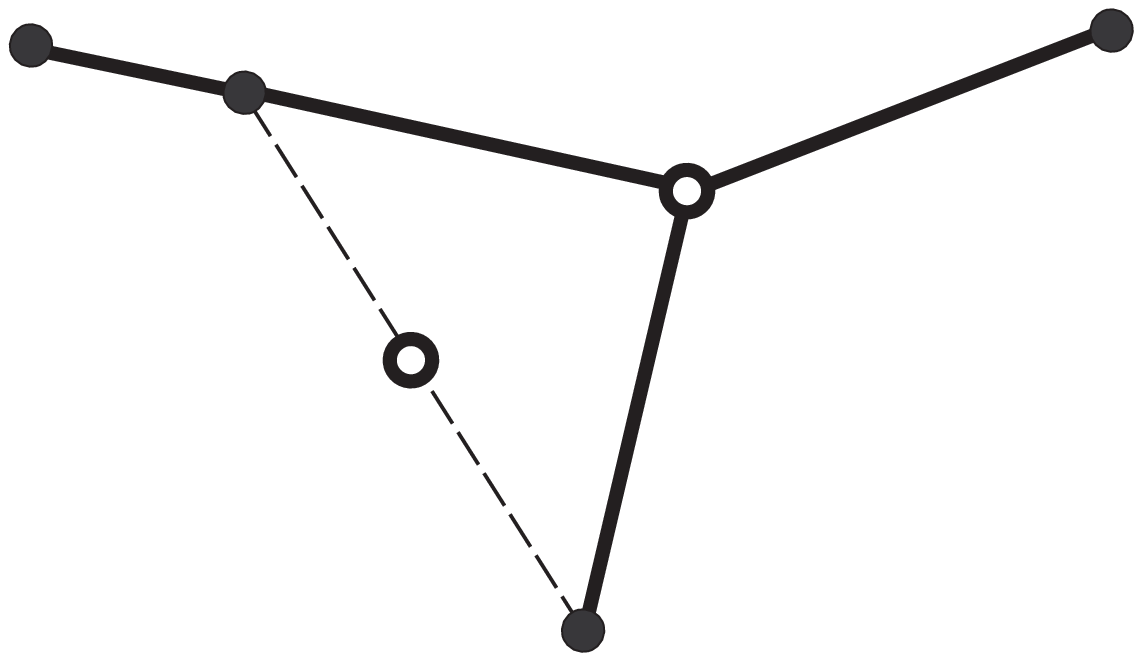}}
    \subfigure[Tree is updated]{\label{fig2c}\includegraphics[scale=0.3]{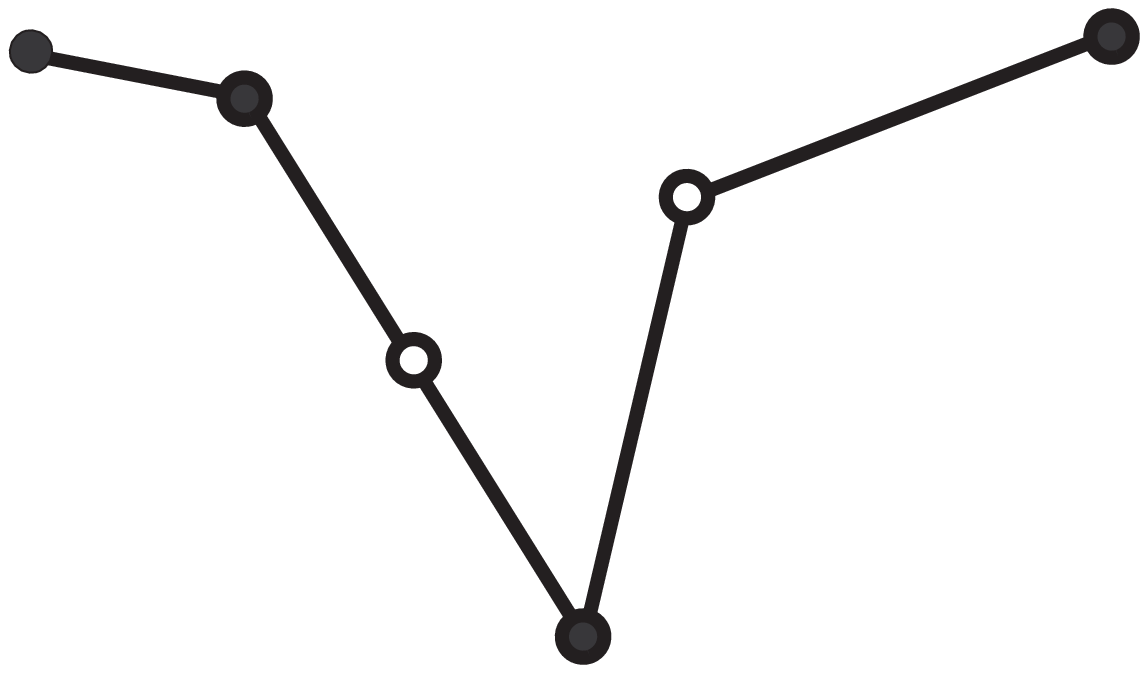}}
    \subfigure[Bent edge is corrected]{\includegraphics[scale=0.3]{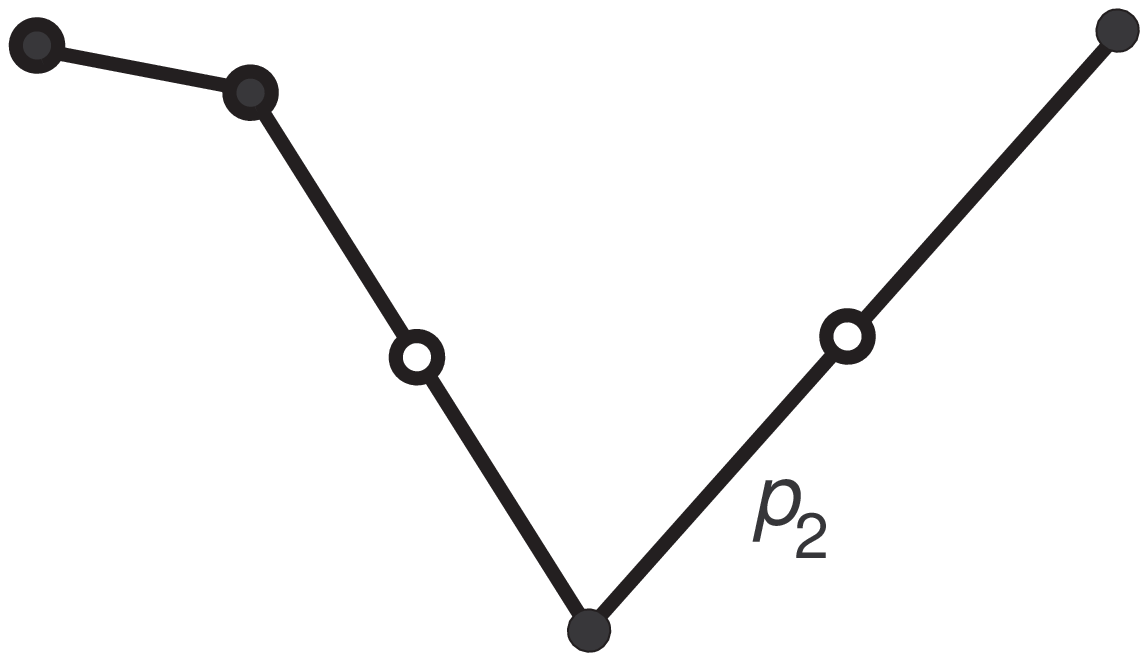}}
  \end{center}
  \caption{The solution constructed by naive I1-BSTH as described in Example \ref{eg1}}
  \label{fig2}
\end{figure}

\clearpage
\begin{figure}[htb]
  \begin{center}
    \subfigure[MST]{\label{figEga}\includegraphics[scale=0.21]{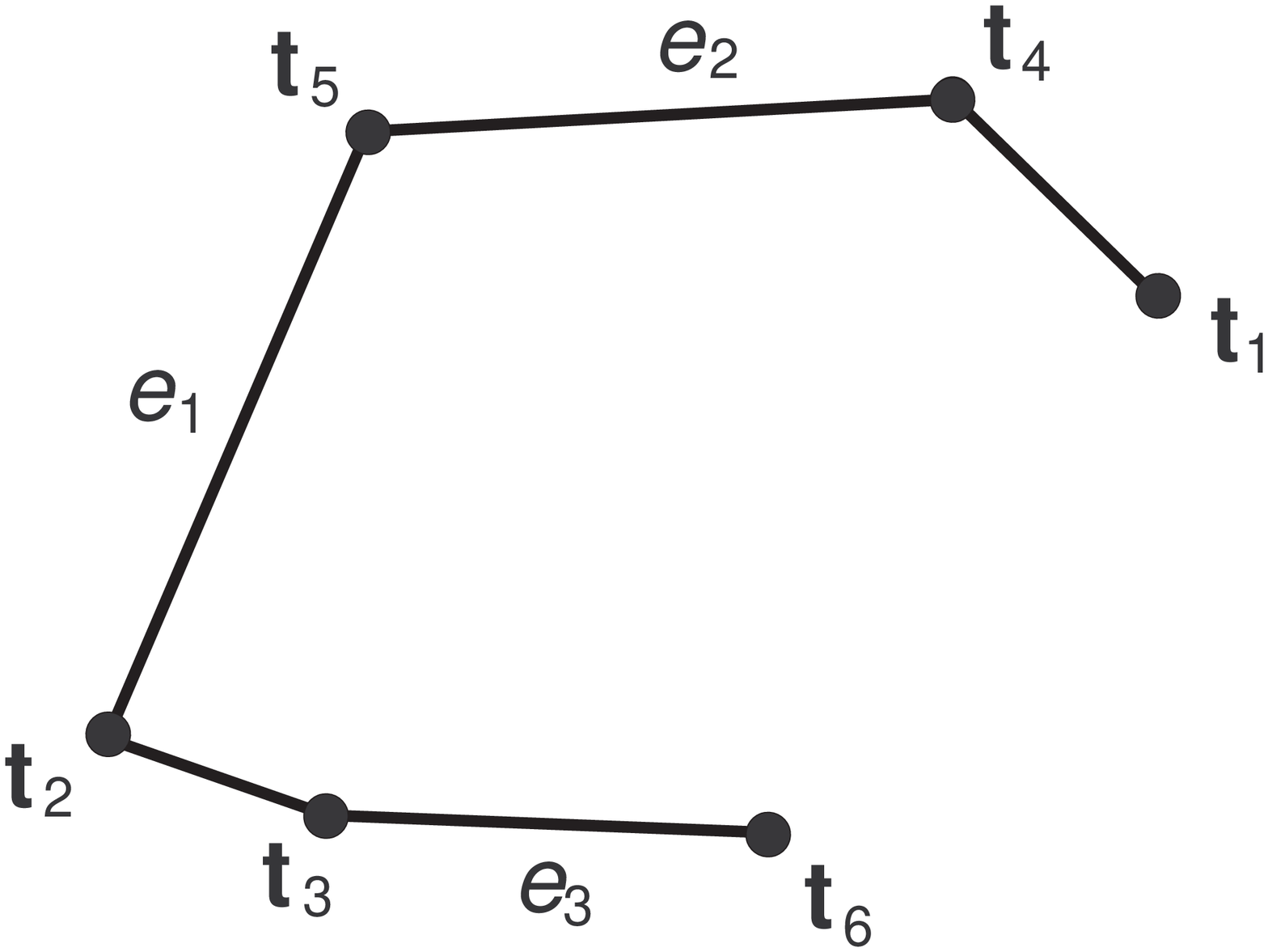}}
    \subfigure[Naive I1-BSTH solution]{\label{figEgb}\includegraphics[scale=0.21]{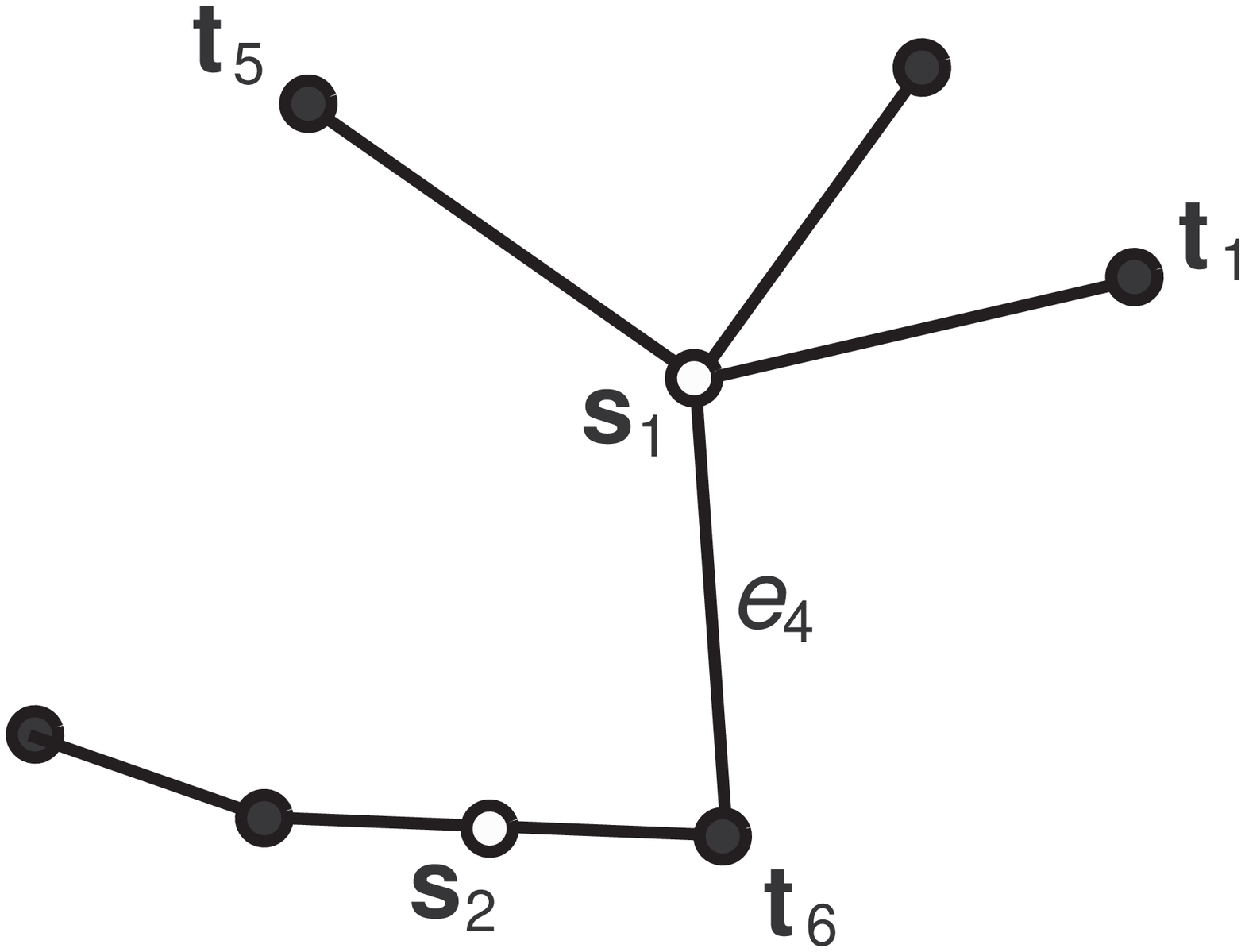}}
    \subfigure[The pre-beaded I1-BSTH solution]{\label{figEgc}\includegraphics[scale=0.21]{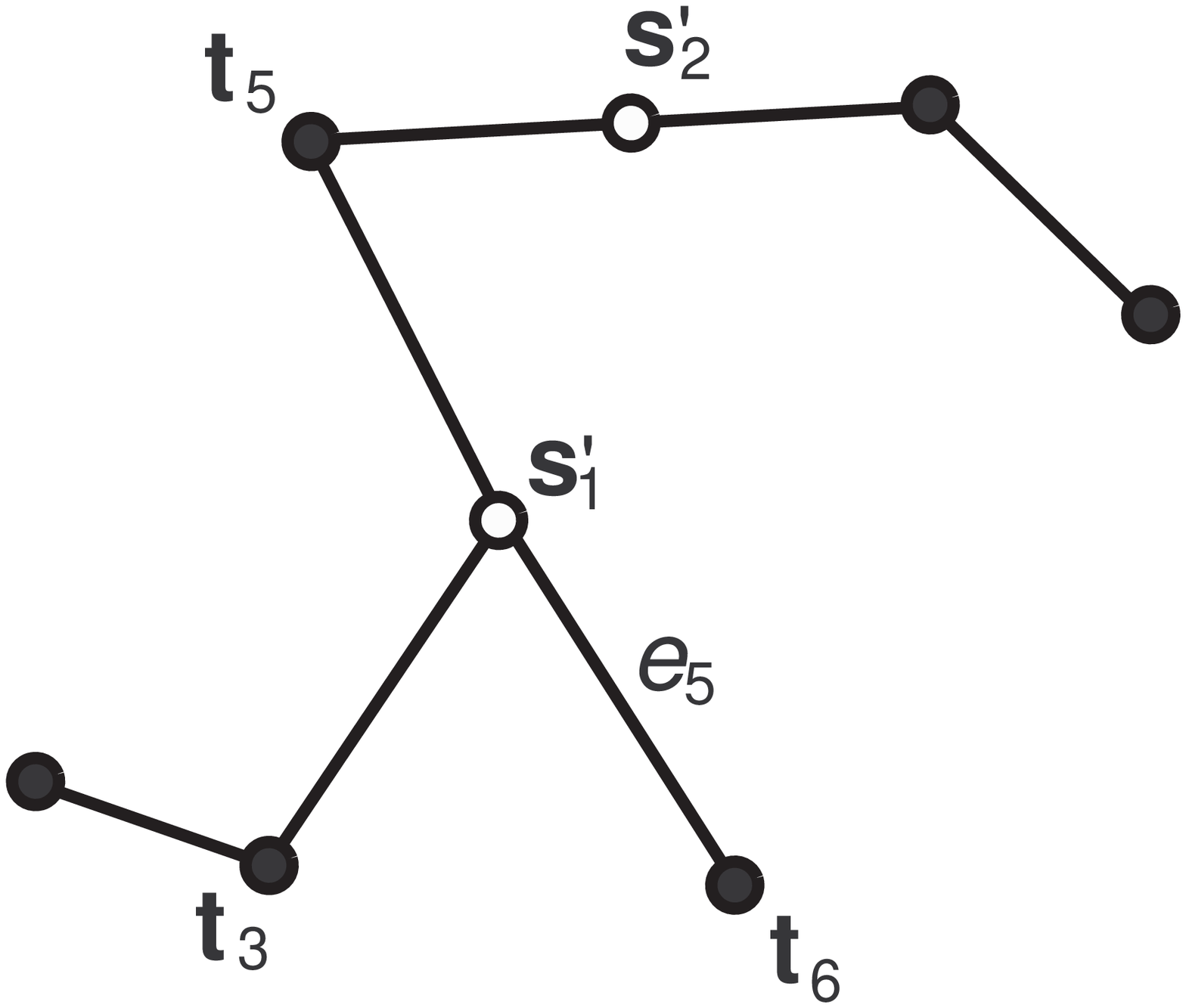}}
    \vspace{5pt}
  \end{center}
  \caption{The benefit of looking ahead}
  \label{figEg1}
\end{figure}

\clearpage
\begin{figure}[htb]
  \begin{center}
    \subfigure[MST]{\label{figEg2a}\includegraphics[scale=0.21]{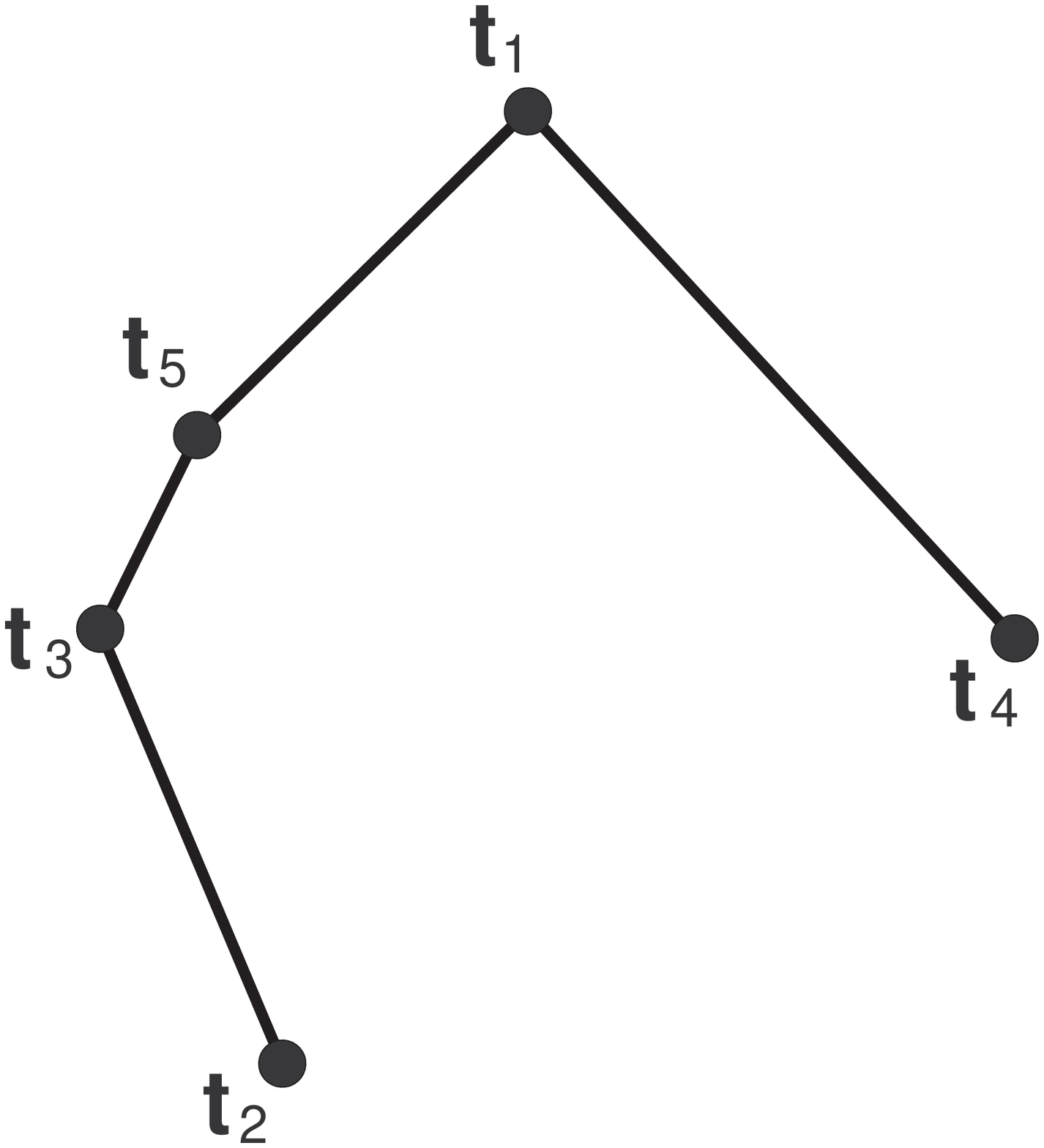}}
    \subfigure[Pre-beaded I1-BSTH solution]{\label{figEg2b}\includegraphics[scale=0.21]{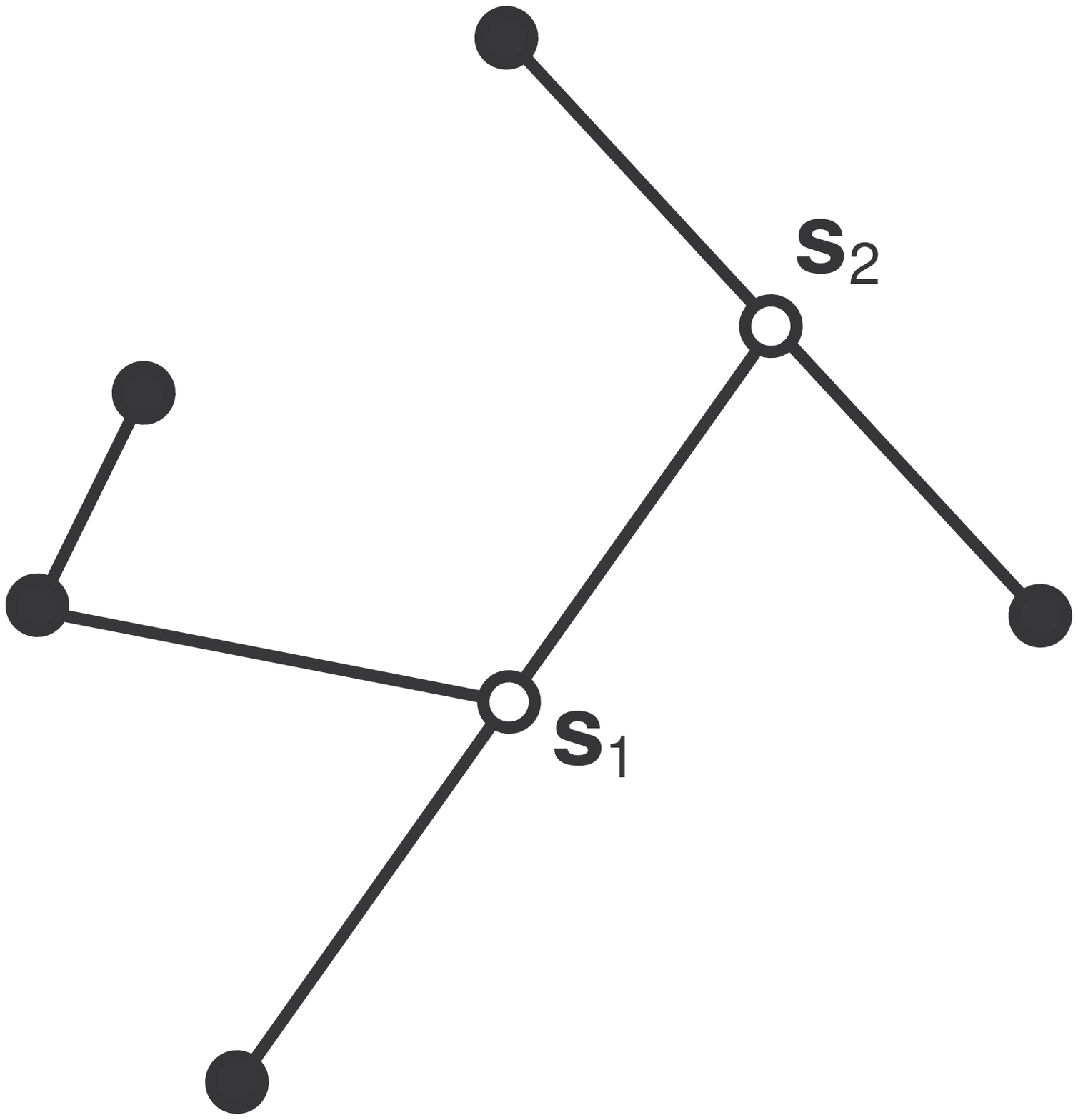}}
  \end{center}
  \caption{The benefit of pre-beading}
  \label{figEg2}
\end{figure}

\clearpage
\begin{figure}[htb]
    \begin{center}
    \includegraphics[scale=0.4]{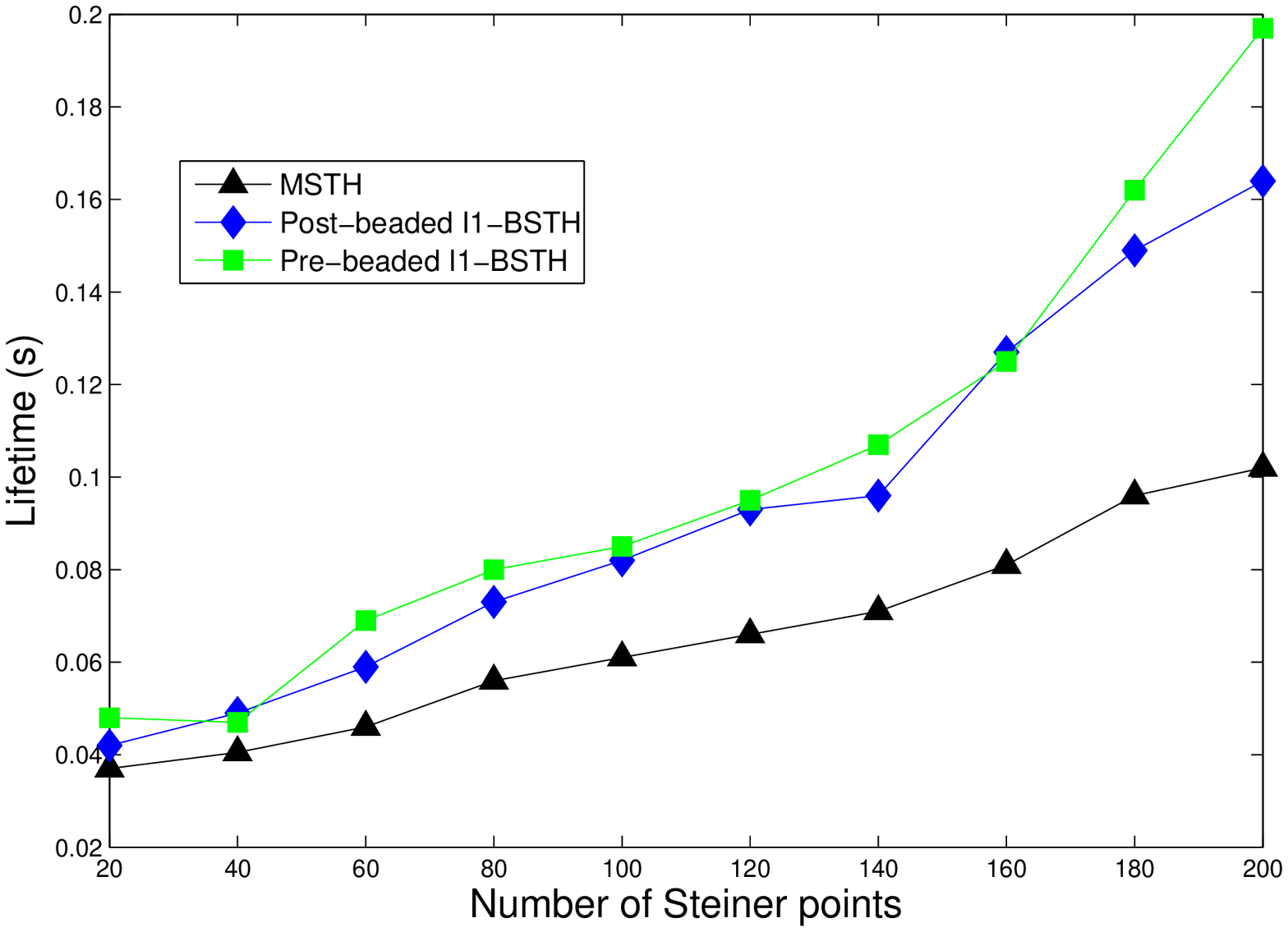}
    \end{center}
    \caption{Comparison of the lifetime of networks constructed by Pre-beaded I1-BSTH, Post-beaded I1-BSTH, and MSTH for a uniform distribution when $\alpha=4$}
    \label{figNew1}
\end{figure}

\clearpage
\begin{figure}[htb]
    \begin{center}
    \includegraphics[scale=0.4]{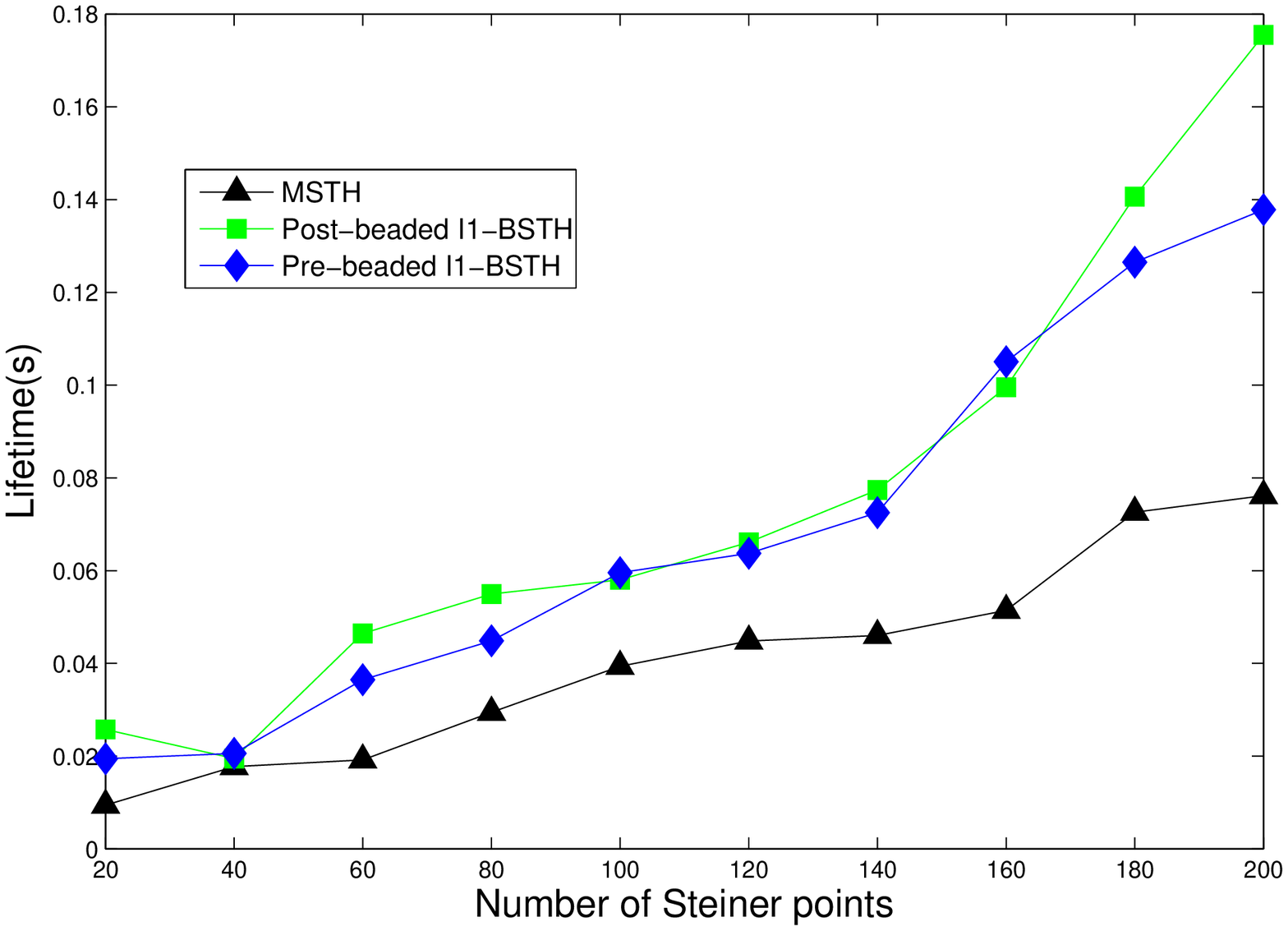}
    \end{center}
    \caption{Comparison of the lifetime of networks constructed by Pre-beaded I1-BSTH, Post-beaded I1-BSTH, and MSTH for a non-uniform distribution when $\alpha=4$}
    \label{figNewD}
\end{figure}

\clearpage
\begin{figure}[htb]
    \begin{center}
    \includegraphics[scale=0.4]{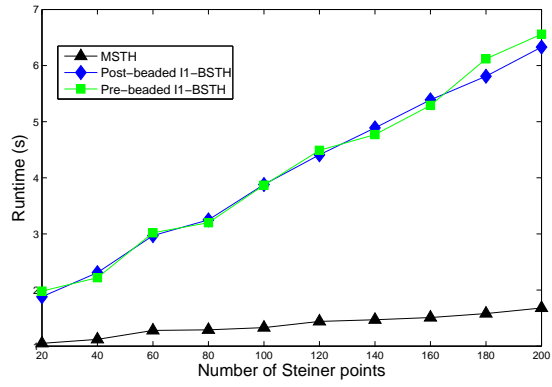}
    \end{center}
    \caption{Comparison of the running times for network construction in Pre-beaded I1-BSTH, Post-beaded I1-BSTH, and MSTH}
    \label{figNew3}
\end{figure}

\clearpage
\begin{figure}[htb]
    \begin{center}
    \includegraphics[scale=0.4]{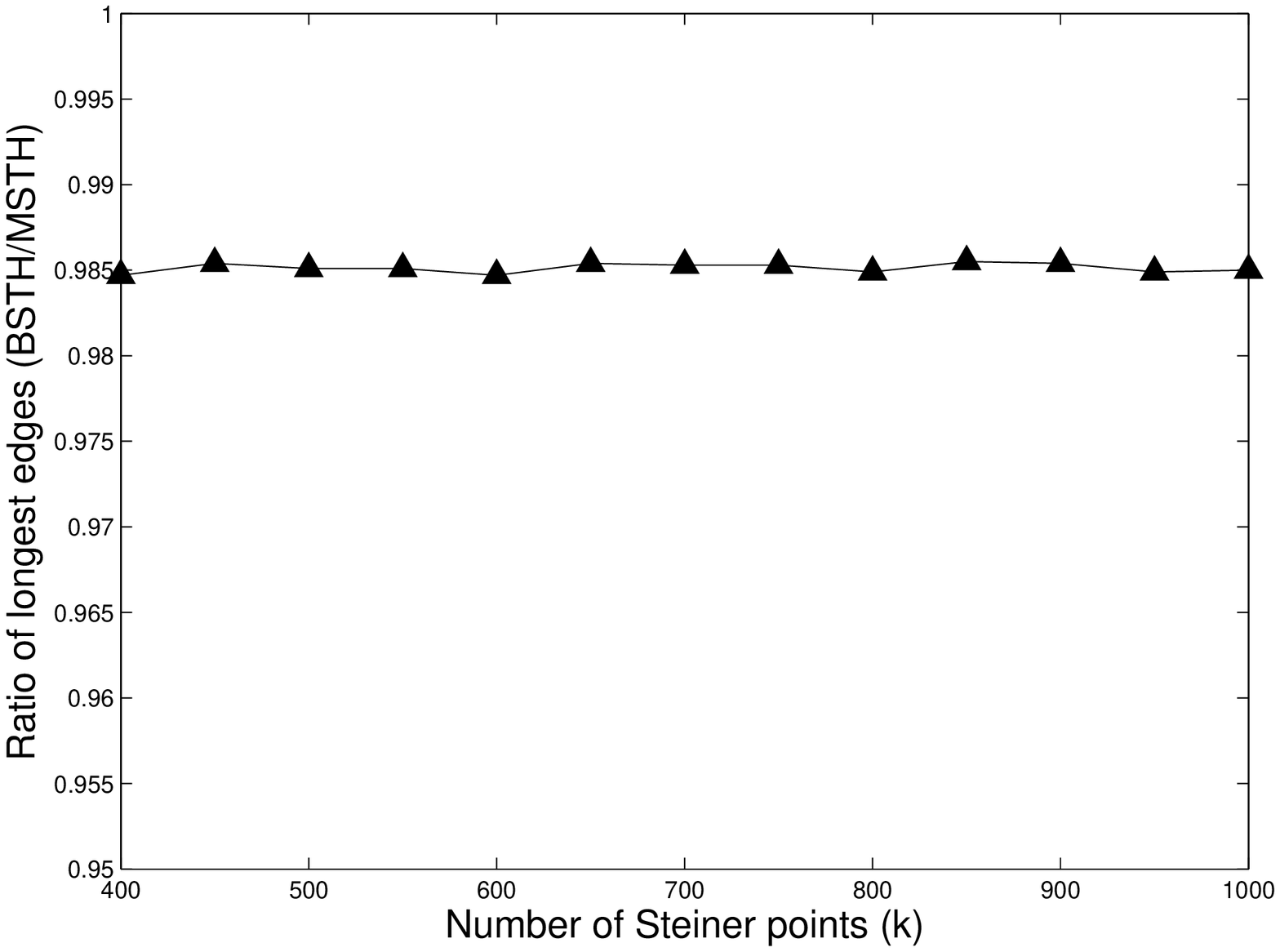}
    \end{center}
    \caption{The ratio $\ell^k_{\mathrm{BSTH}}/ \ell^k_{\mathrm{MSTH}}$ for $n=200$ and $400\leq k\leq 1000$}
    \label{figNew4}
\end{figure}

\begin{thebibliography}{1}
\bibitem{arampatzis}Arampatzis T., Lygeros L., and Manesis J.S.: `A survey of applications of wireless sensors and wireless sensor networks',
{Proc. 13th Mediterranean Conf. Control and Automation}, 2005, pp. 719--724.

\bibitem{mainwaring}Mainwaring A., Polastre J., Szewczyk R., Culler D., and Anderson J.: `Wireless sensor networks for habitat monitoring',
{1st ACM Int. Workshop on Wireless Sensor Networks and Applications}, New York, USA, 2002, pp. 88--97.

\bibitem{cheng}Cheng X., Du D-Z., Wang L., and Xu B.: 'Relay sensor placement in wireless sensor networks', Wireless Networks, 2008, 14, pp. 347--355.

\bibitem{xu}Xu K., Hassanein H., Takahara G., and Wang Q.: `Relay node deployment strategies in heterogeneous wireless sensor networks', IEEE Trans. Mobile Computing, 2010, 9, pp. 145--159.

\bibitem{boulis}Boulis A., Ganeriwal S., and Srivastava M.B.: `Aggregation in sensor networks: an energy-accuracy trade-off', Ad Hoc Networks, 2003, pp. 317--331.

\bibitem{intan}Intanagonwiwat C., Govindan R., and Estrin D.: `Directed diffusion: a scalable and robust communication paradigm for sensor networks', Proc. 6th ACM Annual
Int. Conf. Mobile Computing and Networking, Boston, USA, August 2000, pp. 56-–67.

\bibitem{krish}Krishnamachari L., Estrin D., and Wicker S.: `The impact of data aggregation in wireless sensor networks', Proc. 22nd Int. Con. Distributed
Computing Systems Workshops, Vienna, Austria, July 2002, pp. 575–-578.

\bibitem{lloyd}Lloyd E.L., and Xue G.: `Relay node placement in wireless sensor networks', IEEE Trans. Computers, 2007, 56, pp. 134--138.

\bibitem{brazil2}Brazil M., Ras C.J., and Thomas D.A.: `The bottleneck $2$-connected $k$-Steiner network problem for $k\leq 2$', Discrete Applied Mathematics, 2012, 160, pp. 1028--1038.

\bibitem{brazil3}Brazil M., Ras C.J., and Thomas D.A.: `Approximating minimum Steiner point trees in Minkowski planes', {Networks}, 2010, {56}, pp. 244--254.

\bibitem{bred}Bredin J.L., Demaine E.D., Hajiaghayi M.T., and Rus D., `Deploying sensor nets with guaranteed capacity and fault tolerance', {Proc. 6th ACM Int. Symp. Mobile Ad Hoc Networking and Computing}, New York, USA, 2005, pp. 309--319.

\bibitem{wang}Wang L., and Du D.Z.: `Approximations for a bottleneck Steiner tree problem', {Algorithmica}, 2002, {32}, pp. 554--561.

\bibitem{bae}Bae S.W., Lee C., and Choi S.: `On exact solutions to the Euclidean bottleneck Steiner tree problem',
{Inf. Proc. Letters}, 2010, {110}, pp. 672--678.

\bibitem{boua}Bouabdallah F., Bouabdallah N., and Boutaba R.: `On balancing energy consumption in wireless sensor networks',
{IEEE Trans. Vehicular Technology}, 2009, {58}, pp. 2909--2924.

\bibitem{kalp}Kalpakis K., and Tang S.: `A combinatorial algorithm for the maximum lifetime data gathering with aggregation problem in sensor networks', {Computer Communications}, 2009, {32}, pp. 1655-–1665.

\bibitem{zhang}Zhang H., and Shen H.: `Balancing energy consumption to maximize network lifetime in data-gathering sensor networks',
{IEEE Trans. Parallel and Distributed Systems}, 2009, {20}, pp. 1526--1539.

\bibitem{wang3}Wang Q., and Zhang T.: `Bottleneck zone analysis in energy-constrained wireless sensor networks', {IEEE Commun. Letters}, 2009, {13}, pp. 423--425.

\bibitem{xin}Xin Y., Guven T., and Shayman M.: `Relay deployment and power control for lifetime elongation in sensor networks', {IEEE Int. Conf. Commun.}, 2006, {8}, pp. 3461--3466.

\bibitem{du-2}Du X., Liu X., and Xiao Y.: `Density-varying high-end sensor placement in heterogeneous wireless sensor networks', {IEEE Int. Conf. Commun.}, 2009, pp. 1--6.

\bibitem{li}Li J.S., Kao H.C., and Ke J.D.: `Voronoi-based relay placement scheme for wireless sensor networks', {IET Communications}, 2009, {3}, pp. 530-–538.

\bibitem{wang2}Wang F., Wang D., and Liu J.: `Traffic-aware relay node deployment for data collection in wireless sensor networks',
{Proc. 6th Annual IEEE Commun. Society Conf. Sensor, Mesh and Ad Hoc Communications and Networks}, 2009, pp. 351--359.

\bibitem{sarrafzadeh}Sarrafzadeh M., and Wong C.K.: `Bottleneck Steiner trees in the plane', {IEEE Trans. Computers}, 1992, {41}, pp. 370--374.

\bibitem{drezner}Drezner Z., and Wesolowsky G.O.: `A new method for the multifacility minimax location problem', {Journal of the Operational Research Society}, 1978, {29}, pp. 1095--1101.

\bibitem{love}Love R.F., Wesolowsky G.O., and Kraemer S.A.: `A multifacility minimax location method for Euclidean distances', {Int. Journal of Production Research}, 2009, {11}, pp. 37--45.

\bibitem{du1}Du D.Z., Wang L., and Xu B.: `The Euclidean bottleneck Steiner tree and Steiner tree with minimum number of Steiner points', {7th Annual Int. Conf. Computing and Combinatorics}, Guilin, China, August 2001, pp. 509--518.

\bibitem{bae1}Bae S.W., Choi S., Lee C., and Tanigawa S.: `Exact algorithms for the bottleneck Steiner tree problem', {Algorithmica}, 2011, {61}, pp. 924--947.

\bibitem{brazil}Brazil M., Ras C.J., Swanepoel K., and Thomas D.A.: `Generalised $k$-Steiner tree problems in normed planes', Submitted for publication, arXiv:1111.1464 [math.CO].

\bibitem{lee}Lee J.: `A first course in combinatorial optimization', Cambridge Texts in Applied Mathematics, Cambridge University Press, New York, 2004.

\bibitem{prom}Promel H.J., and Steger A.: `A new approximation algorithm for the Steiner tree problem with performance ratio $5/3$', {Journal of Algorithms}, 2000, {36}, pp. 89--101.

\bibitem{gilbert}Gilbert E.N., and Pollak H.O.: `Steiner minimal trees', {SIAM Journal of Applied Mathematics}, 1968, {16}, pp. 1--29.
\end{thebibliography}
\end{document}